\documentclass[11pt]{article}
\usepackage{amsthm,mathrsfs,amssymb}
\usepackage{baskervald}
\usepackage{amsmath,verbatim}
\usepackage{hyperref}
\usepackage[numbers]{natbib}
\usepackage[active]{srcltx}
\usepackage[matrix,arrow,curve]{xy}
\usepackage{bbm}
\usepackage{verbatim}
\usepackage{setspace}
\usepackage{parskip}
\usepackage[margin=2.5cm]{geometry}
\newtheoremstyle{mytheoremstyle} 
{\topsep}                    
{\topsep}                    
{\itshape}                   
{1.5em}                           
{\itshape\bf}                   
{.---}                          
{.50em}                       
{}  
\theoremstyle{mytheoremstyle}

\newtheorem{thm}{Theorem}
\newtheorem{prop}{Proposition}
\newtheorem{lemma}{Lemma}
\newtheorem{cor}{Corollary}

\newtheorem*{cor*}{Corollary}

\newtheorem*{thm*}{Theorem}

\theoremstyle{definition}
\newtheorem{defn}{Definition}

\numberwithin{equation}{section}
	\usepackage{verbatim}

\newcommand{\calF}{\mathcal{F}}

\newcommand{\frakc}{\mathfrak{c}}

\newcommand{\frakn}{\mathfrak{n}}

\newcommand{\fraks}{\mathfrak{s}}
\newcommand{\frakt}{\mathfrak{t}}

\newcommand{\frakz}{\mathfrak{z}}

\newcommand{\frakH}{\mathfrak{H}}

\newcommand{\bfx}{{\bf x}}
\newcommand{\bfy}{{\bf y}}
\newcommand{\bfz}{{\bf z}}

\newcommand{\CC}{\mathbb{C}}

\newcommand{\QQ}{\mathbb{Q}}
\newcommand{\RR}{\mathbb{R}}

\newcommand{\ZZ}{\mathbb{Z}}

\newcommand{\so}{\Longrightarrow}

\newcommand{\rar}{\rightarrow}

\newcommand{\wh}{\widehat}

\newcommand{\ol}{\overline}

\newcommand{\sgn}{\mathrm{sgn}}

\newcommand{\Sp}{\operatorname{Sp}}

\newcommand{\rank}{\mathrm{rank}}

\newcommand{\Sym}{\mathrm{Sym}}

\renewcommand{\Re}{\operatorname{Re}}

\newcommand{\GL}{{\mathrm{GL}}}

\newcommand{\diag}{\mathrm{diag}}

\newcommand{\mtwo}[1]{\left(\begin{array}{cc} #1 \end{array}\right)}

\renewcommand{\mod}{\text{mod }}

\newcommand{\legendre}[2]{\genfrac{(}{)}{}{}{#1}{#2}}

\newcommand{\floor}[1]{\lfloor #1 \rfloor}

\newlength{\dhatheight}

\title{\textsc{Theta Functions and Reciprocity Laws}}
\author{Z. Amir-Khosravi}
\date{}

\begin{document}
\maketitle

\begin{abstract}In the first part, we consider generalized quadratic Gauss sums as finite analogues of the Jacobi theta function, and the reciprocity law for Gauss sums as their transformation formula. We attach finite Dirichlet series to Gauss sums using a M\"obius transform, and show they have a functional equation, Euler product factorization, and roots only on the critical line. In the second part, we prove a general reciprocity law for sums of exponentials of rational quadratic forms in any number of variables, using the transformation formula of the Riemann theta function on the Siegel upper half-space.
\end{abstract}

\tableofcontents

\section*{Introduction}

Let $a,b,c\in \ZZ$ such that $ac+b$ is even and $ac\neq 0$. The reciprocity law for Gauss sums is the identity:
\begin{align}\label{R} \frac{1}{\sqrt{|c|}}\sum_{x=0}^{|c|-1} \exp\left(\frac{\pi i(ax^2+bx)}{c}\right) = \frac{e^{i\pi \sgn(ac)/4}}{\sqrt{|a|}} \exp( \frac{- \pi ib^2}{4ac}) \sum_{x=0}^{|a|-1} \exp\left(\frac{-\pi i(cx^2+bx)}{a}\right).\end{align}
(e.g. \cite[Theorem 1.2.2]{Ber98})

Setting $t=a/c$, $s=b/2c$, it can be stated as:
\begin{align}\label{Ttrans}T(\frac{s}{t},\frac{-1}{t})=\alpha(s,t)T(s,t) ,\ \ \ \alpha(s,t)=\sqrt{-it} \exp(\frac{s^2\pi i}{t}),\end{align}
where
\begin{align}\label{Tdef} T(s,t) = \frac{1}{2}\sum_{x=0}^{2\delta_t-1} \exp(\pi i x^2 t + 2\pi i x s),\ \ \ s\in \QQ,\ t\in \QQ,\end{align}
and $\delta_t>0$ is the denominator of $t$ in reduced form. We observe that the Jacobi theta function 
$$ \vartheta(z,\tau)=\sum_{x=-\infty}^{\infty} \exp(\pi i x^2\tau + 2\pi i xz),\ \ \ \tau\in \frakH,\ z\in \CC,$$
satisfies exactly the same transformation formula:
\begin{align}\label{thtr} \vartheta(\frac{z}{\tau},\frac{-1}{\tau}) =\alpha(z,\tau)
\vartheta(z,\tau),\ \ \ \alpha(z,\tau)=\sqrt{-i\tau} \exp(\frac{z^2\pi i}{\tau}).\end{align}

In the first part of this article, we investigate the degree to which $T(s,t)$ behaves like a theta function. We start with some basic properties of the function
$$u: \QQ \rar \CC,\ \  u(\frac{a}{b}) = \frac{1}{b}T(0,\frac{a}{b}),\ \ \ \gcd(a,b)=1,$$
which transforms like a weight $-\frac{1}{2}$ modular form. The first main result is as follows.

\textbf{Theorem A.} {\it For $n>0$ even, let
$$ Z_{n}(s) = \sum_{n=d_1 d_2} u(\frac{d_2}{d_1}) \frac{d_1^s}{\sqrt{\gcd(d_1,d_2)}},$$
and for each prime $p$ dividing $n$, put
\begin{align} Z_{n,p}(s) = \sum_{n=p^k m} u(\frac{m}{p^{k}}) \frac{p^{ks}}{\sqrt{\gcd(p^k,m)}}.
\end{align}
Then:
	$$\begin{array}{llc}(a) & L_n(1-s) = e^{\pi i/4} \cdot \ol{L_n(\ol{s})},\ \ \text{where }L_n(s) = n^{-s/2} Z_n(s), \\
	(b) & Z_n(s) = \prod_{p|n\text{ prime }} Z_{n,p}(s),\\
	(c) & Z_n(s)=0\text{ only if }\Re(s)=\frac{1}{2}.
	\end{array}$$}

The proof is elementary. Part $(c)$ amounts to the fact that the quadratic Gauss sum modulo $n$ has absolute value $\sqrt{n}$, which is as expected.

The second part is concerned with generalizing the reciprocity law (\ref{R}). For any $\frakt\in \Sym_n(\QQ)$ we can write $\frakt = UDV^{-1}$, where $U$, $V \in \GL_n(\ZZ)$ and 
$$D=PQ^{-1},\ \ \ P=\diag(p_1,\cdots, p_n),\ \ Q=\diag(q_1,\cdots, q_n),\ \ p_i,\ q_i\in \ZZ\ \ \gcd(p_i,q_i)=1.$$
We say $\frakt$ is in \textit{reduced form} when written $\frakt=AB^{-1}$, for $A = UP$, $B=VQ$, and $U$, $V$, $P$, $Q$ as above.

\textbf{Theorem B.} {\it Let $\frakt=AB^{-1}\in \Sym_n(\QQ)$ be in reduced form and non-singular. Let $\fraks\in \frac{1}{2}\ZZ^n$ be such that ${}^tB A+2\diag({}^t B\fraks)$ is integral and even. Then
	$$ \frac{1}{\sqrt{|\det(B)|}}\sum_{{\bf x}\in \ZZ^n(\mod B)} e(\frac{1}{2}{}^t {\bf x} \frakt {\bf x} + {}^t {\bf x} \fraks)= \frac{e^{\pi i\sigma_{\frakt}/4}}{\sqrt{|\det(A)|}}\exp(-\pi i {}^t\fraks \frakt^{-1} \fraks) \sum_{{\bf x}\in \ZZ^n(\mod A)} e(-\frac{1}{2}{}^t{\bf x} \frakt^{-1} {\bf x} + {}^t{\bf x} \fraks).$$}

Here $e(x)=\exp(2\pi i x)$ and $\sigma_{\frakt}$ is the signature of $\frakt$. The case $n=1$ is then (\ref{R}). The proof adapts the classical analytic argument using the Jacobi theta function to one using the Riemann theta function on the Siegel upper half-space of degree $n$.

\subsection*{Notation}

We shall use the number theorist's exponential throughout:
\begin{align} e(x) = e^{2\pi i x}.\end{align}
We also fix the symbol
\begin{align} w = e^{\pi i/4}.\end{align}
For $x>0$, $\sqrt{-x}$ will always mean $i\sqrt{|x|}$.

\section{Gauss Sums as Finite Theta Functions}

Let $n>0$ be an even integer. We first we show how the reciprocity law for generalized quadratic Gauss sums amounts to the Poisson summation formula on $\ZZ/n\ZZ$, plus the determination of the sign of the quadratic Gauss sum modulo $n$. Next, we define a finite zeta function $Z_n(s)$ as a certain transform of the Gauss sum modulo $n$, and show it has properties similar to global zeta functions.

\subsection{Poisson Summation on Cyclic Groups}

For $n>0$, we put
\begin{align} e_n(x) = e(\frac{x}{n})=e^{\frac{2\pi ix}{n}}.\end{align}

Let $G_n = \ZZ/n\ZZ$. We fix the Fourier transform of a function $f: G_n \rar \CC$ as
$$\calF_n f (x) = \sum_{y\in G_n} f(y)e_n(-xy),$$
and also denote it by $\wh{f}$, when $n$ is understood. The Fourier inversion formula is 
\begin{align} \label{Fif} \calF_n (\calF_n f) (x) = n\cdot f(-x).\end{align}

For $a|n$, let $H_a\subset G_n$ be the subgroup of order $a$. If $n=ab$, the finite version of the Poisson summation formula is:
\begin{align}\label{PSF} \frac{1}{\sqrt{|H_a|}} \sum_{x\in H_a} f(x) = \frac{1}{\sqrt{|H_b|}} \sum_{x\in H_b} \wh{f} (x).\end{align}
Both (\ref{PSF}) and (\ref{Fif}) are easy to verify. 

For $a,b,c\in \ZZ$, $ac\neq 0$, put
$$ f_{a,b,c}(x) = e(\frac{ax^2+bx}{2c}).$$
Then $f_{a,b,c}(x+c) = f_{a,b,c}(x)(-1)^{ac+b}$, so that $f_{a,b,c}: G_c \rar \CC$ is well-defined if and only if $ac+b$ is even. In that case we put
\begin{align}S_{a,b,c} = \frac{1}{\sqrt{|c|}} \sum_{x\in G_c} f_{a,b,c}(x).\end{align}

For $n>0$, and $m\in \ZZ$ such that $mn$ is even, put
	$$ C(m,n)=\left\{\begin{array}{lc} S_{1,0,n} &m\text{ even,}\\ \\S_{1,1,n}\cdot e(\frac{1}{8n})&m\text{ odd.}\end{array}\right.$$

\begin{lemma}\label{whf}For $n>0$, $m\in \ZZ$ such that $mn$ is even,
	$$ \wh{f}_{1,m,n}(x) =C(m,n) e(\frac{-m^2}{8n}) f_{-1,m,n}(x).$$ 
\end{lemma}
\begin{proof}
	If $m=2l$,
	\begin{align*}
	\wh{f}_{1,m,n}(x) &= \sum_{y\in G_n} e_{2n}((y^2+2y(l-x)) = \left(\sum_{y\in G_n} e_{2n}((y+l-x)^2)\right)e_{2n}(-(l-x)^2)\\
	&=(S_{1,0,n})e_{2n}(-x^2+2lx) e_{2n}(-l^2) = (S_{1,0,n}) f_{-1,m,n}(x) e(\frac{-m^2}{8n}).
	\end{align*}
	If $m=2l+1$,
	\begin{align*}
	\wh{f}_{1,m,n}(x) &= \sum_{y\in G_n} e_{2n}( y^2+2(l-x)y+y) = \left(\sum_{y\in G_n} e_{2n}((y+(l-x))^2+(y+l-x))\right)e_{2n}(-(l-x)^2 - (l-x))\\
	&=(S_{1,1,n})e_{2n}(-x^2+2lx+x) e_{2n}(-l^2-l) =(S_{1,1,n})f_{-1,m,n}(x) e(\frac{-m^2+1}{8n}).
	\end{align*}
\end{proof}
\begin{prop}Let $a,b,c\in \ZZ$, with $ac\neq 0$, and $ac+b$ even. Then
	$$S_{a,b,c} =  C(b,ac) e(-\frac{b^2}{8ac}) S_{-c,b,a}.$$
\end{prop}
\begin{proof}
Note that
$$f_{a,b,c}(x) = e\left( \frac{ (ax)^2 + (ax)b}{2ac}\right) = f_{1,b,ac}(ax).$$ 
Therefore
$$ S_{a,b,c} = \frac{1}{\sqrt{|c|}}\sum_{x=0}^{|c|-1} f_{1,b,ac}(ax) = \frac{1}{\sqrt{|H_c|}} \sum_{x\in H_c} f_{1,b,ac}(x),$$
and
$$S_{-c,b,a} = \frac{1}{\sqrt{|a|}}\sum_{x=0}^{|a|-1} f_{-1,b,ac}(cx) = \frac{1}{\sqrt{|H_a|}} \sum_{x\in H_a} f_{-1,b,ac}(x).$$
Since $\wh{f}_{1,b,ac}(x)=C(b,ac)e(-\frac{b^2}{8ac})f_{-1,b,ac}(x)$ by Lemma \ref{whf}, the Proposition follows from the Poisson summation formula \ref{PSF}.
\end{proof}
The reciprocity law (\ref{R}) can also be written as
\begin{align}\label{SR} S_{a,b,c} = e(\frac{\sgn(ac)}{8}) e(-\frac{b^2}{8ac}) S_{-c,b,a}.\end{align}
In view of the proposition, it's equivalent to the identity
$$ C(b,ac)=e(\frac{\sgn(ac)}{8}),\ \ \ ac\text{ even,}$$
or the pair of identities
$$ S_{1,0,n} = e(\frac{1}{8}),\ \ \ S_{1,1,n} = e(\frac{1}{8}-\frac{1}{8n}),$$
for $n>0$.

\subsection{A finite zeta function}

For $t,s\in \QQ$. For $t\neq 0$, let $\delta_t$ denote the denominator of $t$ in reduced form, recall from the introduction
\begin{align} T(s,t) = \frac{1}{2}\sum_{x=0}^{2\delta_t-1}\exp(\pi i tx^2 + 2\pi i sx).\end{align}
If $t=0$, we take $T(s,0)=0$, so that $T(s,t)=T(s,t+2)$ for all $s,t$. As mentioned in the introduction, the reciprocity law takes the same form
\begin{align}\label{Tid}
T(\frac{-1}{t},\frac{s}{t})=\alpha(s,t)T(s,t) =  ,\ \ \ \alpha(s,t) = \sqrt{-it} \exp(\pi i \frac{s^2}{t}).
\end{align}

For integers $p,q$, with $q>0$, we put
$$ U(p,q) = \frac{1}{q}T(0,\frac{p}{q})=\frac{1}{2q} \sum_{x=0}^{2q-1} e\left(\frac{px^2}{2q}\right).$$

\begin{lemma}\label{Ulem}With $p,q$ as above, \begin{itemize}	
		\item[(a)]$U(p,q)=0$ if $pq$ is odd.
		\item[(b)] $U(pr,qr)=U(p,q)$ for $r\neq 0$.
		\item[(c)]$U(pr^2,q)=U(p,q)$ if $pq$ is even and $\gcd(r,q)=1$.
	\end{itemize}
\end{lemma}
\begin{proof}
	Both $(a)$ and $(b)$ follow from
	$$ e\left(\frac{p(x+q)^2}{q}\right)=e\left(\frac{px^2}{q}\right) (-1)^{pq}.$$
	Part $(c)$ follows from writing
	$$ U(p,q)=\frac{1}{q}\sum_{x\in \ZZ/q\ZZ} e\left(\frac{px^2}{q}\right),$$
	and making the substitution $x\rar rx$.	
\end{proof}
For $r=\frac{p}{q}$ a non-zero rational number, we put
\begin{align}\label{udef} u(r)=U(p,q),\end{align}
which is well-defined by part $(b)$ of the lemma. If $pq$ is odd, then $u(r)=0$ by $(a)$.

\begin{prop}
	\label{uprop1}The function $u:\QQ \rar \CC$ satisfies the following properties.
	\begin{itemize}	 
		\item[(a)] Suppose $a,b,c,d\in\ZZ$, with $b$, $d>0$, $\gcd(b,d)=1$, and $ab$, $cd$ even. Then
		$$ u(\frac{a}{b} + \frac{c}{d})=u(\frac{a}{b}) \cdot u(\frac{c}{d}).$$ 	
		\item[(b)] Suppose $pqr$ is non-zero and even, and $\gcd(q,r)=1$. Then
		$$ u(\frac{pq}{r}) u(\frac{pr}{q}) = u(\frac{p}{qr}).$$
	\end{itemize}
\end{prop}
\begin{proof}
	For $(a)$, we write
	$$ u(\frac{a}{b}+\frac{c}{d}) =\frac{1}{bd}\sum_{z\in \ZZ/bd\ZZ} e\left(\frac{(ad+bc)z^2}{2bd}\right).$$
	Fix $r$, $s$ integers such that $rb+sd=1$. Then each $z\in\ZZ/bd\ZZ$ is uniquely of the form $xsd+yrb$ for $(x,y)\in \ZZ/b\ZZ \times \ZZ/d\ZZ$. We have	$$ u(\frac{a}{b}+\frac{c}{d})=\frac{1}{bd}\sum_{x=0}^{b-1} \sum_{y=0}^{d-1} e\left(\frac{(ad+bc)(xsd+yrd)^2}{2bd}\right).$$
	Now we have
	$$ (xsd+yrb)^2 \equiv x^2s^2d^2 + y^2r^2b^2\ (\mod 2bd).$$
	It follows that
	\begin{align*} (ad+bc)(xsd+yrd)^2&\equiv (ad+bc)(x^2s^2d^2+y^2r^2b^2) \equiv x^2 as^2d^3 + y^2cr^2b^3 + bd(ab y^2 r^2 + cd x^2s^2)\\
	&\equiv x^2as^2d^3 + y^2cr^2 b^3\ (\mod 2bd),
	\end{align*}
	using the fact that $ab$ and $cd$ are even.	On the other hand from $rb+sd=1$, we have $r^2b^2 + s^2d^2 \equiv 1\ (\mod 2bd)$. From this and the above we obtain
	\begin{align*} (ad+bc)(xsd+yrd)^2 &\equiv ad x^2 (1 - r^2b^2) + bcy^2(1-s^2d^2)\equiv adx^2 + bcy^2 - bd (abr^2 + cds^2)\\
	&\equiv adx^2 + bcy^2\ (\mod 2bd),\end{align*}
	again using that $ad$, $bc$ are even.  Therefore
	$$ u(\frac{a}{b}+\frac{c}{d}) = \frac{1}{bd} \sum_{x=0}^{b-1} \sum_{y=0}^{d-1} e\left(\frac{adx^2+bcy^2}{2bd}\right) = u(\frac{a}{b})u(\frac{c}{d})$$
	proving $(a)$. 
	
	For $(b)$, first note that if $p$, $q$, $r$ are all odd, both sides are zero by Lemma \ref{Ulem}(a). If $pqr$ is even, we have
	$$ u(\frac{pq}{r})u(\frac{pr}{q})=u(\frac{p(q^2+r^2)}{rq})=u(\frac{p(q+r)^2}{rq}-2)=u(\frac{p(q+r)^2}{rq})$$
	using $(a)$, and the fact that $u$ has period $2$. Since $\gcd(rq,q+r)=1$, we have
	$$ u(\frac{p(q+r)^2}{rq})=u(\frac{p}{qr})$$
	by Lemma \ref{Ulem}(b), as desired.
\end{proof}

Form this point on until the end of the first section, we assume the general determination of the sign of the Gauss sum, which is the following statement:

Let 
$$ g(a,n) = \sum_{x=0}^{n-1} e(\frac{ax^2}{n}).$$
For $a>0$ odd, put
$$ \epsilon_a = \left\{\begin{array}{cc}1&a\equiv 1(\mod 4)\\
i& a\equiv 3(\mod 4).\end{array}\right.$$
If $a,n>0$ and $\gcd(a,n)=1$, then
\begin{align}\label{SGS} \frac{1}{\sqrt{n}} g(a,n) = \left\{\begin{array}{rr}\legendre{a}{n}& n\equiv 1(\mod 4),\\
0& n\equiv 2(\mod 4),\\
\legendre{a}{n}i&n\equiv 3(\mod 4),\\
\legendre{a}{n}(1+i^a)& n\equiv 0(\mod 4).\end{array}\right.
\end{align}
See \cite[Theorems 1.5.1,2,4]{Ber98}. The case $n=1,3$ is sometime stated together as 
$$g(a,n) = \legendre{a}{n} \epsilon_a \sqrt{n}.$$

\begin{prop}\label{uprop2}The function $u$ satisfies the following properties.
	\begin{itemize}	
		\item[(a)] For $p,q$ arbitrary integers, with $q\neq 0$, we have
		$$ u(\frac{p}{q}) = \sqrt{\frac{p}{q}}\cdot e^{\pi i/4} \cdot u(\frac{-q}{p}).$$ 
		
		\item[(b)] Suppose $p$, $q$ are coprime, $pq$ is even, and $q>0$. Then
		$$ u(\frac{p}{q}) = \left\{\begin{array}{lc}\frac{w^p}{\sqrt{q}} \legendre{q}{p}& \text{ if } p \text{ odd,}\\ \\\
		\frac{w^{1-q}}{\sqrt{q}} \legendre{p}{q}& \text{ if }p \text{ even.}\end{array}\right.$$
		where $w = e^{\pi i/4}$, and $\legendre{p}{q}$ is the Jacobi symbol.
	\end{itemize}
\end{prop}
\begin{proof}$(a)$ follows from reciprocity. For $(b)$, first assume $q=2q_0$ is even. We have $$u(\frac{p}{q})=\frac{1}{4q_0}g(p,4q_0)=\frac{w}{4q_0} \epsilon_{p}^{-1} \sqrt{4q_0} \legendre{4q_0}{p} = \frac{w^p}{\sqrt{2q_0}}\legendre{2q_0}{p},$$
	using (\ref{SGS}), and the observation that
	$$ w^{p-1}\epsilon_p = (-1)^{\frac{p^2-1}{8}}=\legendre{2}{p}.$$
	If $q$ is odd, $p$ is even, and the statement then follows from $(a)$ and the even case applied to $u(\frac{q}{p})$. This proves $(b)$.
\end{proof}

\begin{cor}\label{upmcor}Let $p$ be a prime, $m$ a positive integer such that $mp$ is even and $p\nmid m$. Let $r>0$. Then
	$$ u(\frac{m}{p^r})= \frac{\epsilon(m,p,r)}{\sqrt{p^r}},$$
	where 
	$$\epsilon(m,p,r) =\left\{\begin{array}{cc} w^{1-p} \legendre{m}{p}&p \text{ odd,}\ r\text{ odd,}\\\\
	1&p\text{ odd},\ r\text{ even,}\\\\
	w^m\legendre{2}{m}&p=2,\ r\text{ odd},\\\\
	w^m& p=2,\ r\text{ even}.
	\end{array}\right.$$
\end{cor}
\begin{proof}By part $(a)$ we have
	$$ \sqrt{p^r} u(\frac{m}{p^r}) = \sqrt{m} \cdot w \cdot \ol{u(\frac{p^r}{m})}.$$
	If $p$ is odd, by part $(c)$ we have either $u(\frac{p^r}{m})=u(\frac{p}{m})$ or $u(\frac{1}{m})$, according to whether $r$ is odd or even. The latter case is $\frac{w}{\sqrt{m}}$ by $(a)$, and the former $\frac{w^p}{\sqrt{m}}\legendre{m}{p}$ by $(d)$ and the claim follows either way. If $p=2$ and $r$ is odd, $u(\frac{p^r}{m})=u(\frac{2}{m})=\frac{w^{1-m}}{\sqrt{m}} \legendre{2}{m}$ again by $(d)$, and if $r$ is even, $u(\frac{p^r}{m})=u(\frac{4}{m})=\frac{w^{1-m}}{\sqrt{m}}\legendre{4}{m} = \frac{w^{1-m}}{\sqrt{m}}.$ Again in either case, the claim follows.

\end{proof}

For $r\in \QQ^\times$, the function $u$ satisfies the functional equations
$$ u(\frac{-1}{r}) = \frac{u(r)}{\sqrt{r}},\ \ \ u(r+2)=u(r).$$
In other words, $u$ is a function defined on $\QQ^\times=P^1(\QQ)-\{0,\infty\}$ that transforms the same way as a modular form of weight $-1/2$ for the level subgroup $\Gamma_0(2)$.

For $n>0$ an even integer, and $s\in \CC$, let
\begin{align} Z_n(s) = \sum_{d|n} u(\frac{d'}{d})\frac{d^s}{\sqrt{\gcd(d',d)}},\ \ \ d'=\frac{n}{d}.\end{align}
For each prime divisor $p$ of $n$, let
\begin{align}\label{Znp} Z_{n,p}(s) = \sum_{k=0}^{\alpha_p} u(\frac{n}{p^{2k}}) p^{ks-\frac{\mu(k)}{2}},\ \ \ \mu(k)=\min(k,\alpha_p-k),\end{align}
where $\alpha_p$ is the exponent of $p$ in the factorization of $n$.

\begin{prop}\label{Znpfac}Let $p$ be a prime, and $n=p^{\alpha}\cdot m$ where $p\nmid m$. Writing $\beta=\floor{\frac{\alpha}{2}}$,
	$$(1-p^{s-\frac{1}{2}})Z_{n,p}(s) = \left\{\begin{array}{cl}
	(1-p^{(\beta+1)(s-\frac{1}{2})})(1+w^{1-p}\legendre{m}{p} p^{(\beta+1)(s-\frac{1}{2})})&p\text{ odd, }\alpha\text{ odd,}\\\\
	1-p^{(\alpha+1)(s-\frac{1}{2})}& p\text{ odd, }\alpha\text{ even,}\\\\ 
	(1+2^{(\beta+1)(s-\frac{1}{2})})(1+w^m \legendre{2}{m}2^{(\beta+1)(s-\frac{1}{2})}),&p=2,\ \alpha\text{ odd,}\\\\
	(1-2^{\beta(s-\frac{1}{2})})(1+w^m 2^{(\beta+1)(s-\frac{1}{2})}),&p=2,\ \alpha\text{ even.}
	\end{array}\right.$$
\end{prop}
\begin{proof} First, suppose we're not in the last case where $p=2$ and $\alpha$ is even. Then
	\begin{align*}Z_{n,p}(s) &= \sum_{k=0}^{\beta} u(p^{\alpha-2k}\cdot m) p^{ks-\frac{k}{s}} + \sum_{k=\beta+1}^\alpha u(\frac{m}{p^{2k-\alpha}})p^{ks-\frac{\alpha-k}{2}}\\
	&=\sum_{k=0}^{\beta} p^{k(s-\frac{1}{2})} + \sum_{k=\beta+1}^{\alpha} \epsilon(p,m,k) \cdot p^{-k+\frac{\alpha}{2} + ks - \frac{\alpha-k}{2}}\\
	&=(1-p^{s-\frac{1}{2}})^{-1}\left((1-p^{(\beta+1)(s-\frac{1}{2})}) + \epsilon(p,m,\alpha) (p^{(\beta+1)(s-\frac{1}{2})} - p^{(\alpha+1)(s-\frac{1}{2})})\right)\\,
	\end{align*}	
	using Corollary \ref{upmcor}, and the fact that $\epsilon(p,m,2k-\alpha)=\epsilon(p,m,\alpha)$. If $p$ is odd and $\alpha$ even, $\epsilon(p,m,\alpha)=1$, and the claim follows. If $\alpha$ is odd, $\alpha+1 = 2(\beta+1)$, so
	$$Z_{n,p}(s)=(1-p^{s-\frac{1}{2}})^{-1}(1-p^{(\beta+1)(s-\frac{1}{2})})(1+\epsilon_{p,m,k}p^{(\beta+1)(s-\frac{1}{2})})).$$
	If $p=2$ and $\alpha=2\beta$, the term $u(p^{\alpha-2k}\cdot m)$ vanishes for $k=\beta$, rather than equal $1$ as above. Then
	\begin{align*} Z_{n,2}(s) &= \sum_{k=0}^{\beta-1} 2^{k(s-\frac{1}{2})} + \sum_{k=\beta+1}^{\alpha} \epsilon_{2,m,\alpha} 2^{k(s-\frac{1}{2})} = (1-2^{s-\frac{1}{2}}) \left( (1-2^{\beta(s-\frac{1}{2})}) + w^m (2^{(\beta+1)(s-\frac{1}{2})} - 2^{(\alpha +1)(s-\frac{1}{2})})\right)\\
	&= (1-2^{s-\frac{1}{2}})(1-2^{\beta(s-\frac{1}{2})})(1+w^{m}2^{(\beta+1)(s-\frac{1}{2})}).
	\end{align*}
\end{proof}

\begin{thm}\label{thmA} The functions $Z_n(s)$ satisfy the following properties.
	$$\begin{array}{llc}(a) & L_n(1-s) = w\cdot \ol{L_n(\ol{s})},\ \ \text{where }L_n(s) = n^{-s/2} Z_n(s), \\
	(b) & Z_n(s) = \prod_{ p|n\text{ prime }} Z_{n,p}(s),\\
	(c) & Z_n(s)=0\text{ only if }\Re(s)=\frac{1}{2}.
	\end{array}$$
\end{thm}
\begin{proof}
	By the reciprocity law, we have
	\begin{align*} \ol{Z(\ol{s})} &=  \sum_{n=d_1d_2} w^{-1}\frac{d_2}{\sqrt{n}}u(\frac{d_1}{d_2}) \frac{(n/d_2)^s}{\gcd(d_1,d_2)}= n^{s-\frac{1}{2}} \cdot w^{-1} \cdot \sum_{n=d_1 d_2} u(\frac{d_2}{d_1}) \frac{(d_2)^{1-s}}{\gcd(d_2,d_1)}\\
	&= n^{\frac{s}{2}-\frac{1-s}{2}}\cdot w^{-1} \cdot Z(1-s),\end{align*}
	which proves $(a)$.  For $(b)$, we first write
	$$ Z_n(s) = \sum_{d|n} f_n^{(s)}(d),\ \ \ f_n^{(s)}(d) = u(\frac{n}{d^2}) \frac{d^s}{\gcd(d,n/d)}.$$
	We then check that $f_n^{(s)}(d_1d_2)=f_n^{(s)}(d_1)f_n(d_2)$ if $d_1$, $d_1$ are coprime divisors of $n$. Indeed, writing $n=d_1d_2m$, we have
	$$ \gcd(d_1d_2,m)=\gcd(d_1,d_2m)\gcd(d_2,d_1m),$$
	and from Prop. \ref{uprop1}(b) applied to $p=m$, $q=d_1$, $r=d_2$, we have
	$$ u(\frac{n}{d_1^2})u(\frac{n}{d_2^2}) = u(\frac{n}{d_1^2d_2^2}).$$
	It follows that, letting $\alpha_p$ denote the exponent of $p$ in the factorization of $n$, 
	$$ Z_n(s) = \prod_{p|n\text{ prime}} \sum_{k=0}^{\alpha_p} u(\frac{n/p^{k}}{p^k}) \frac{p^{ks}}{\gcd(p^k,n/p^k)} = \prod_{{ p|n \text{ prime}}}Z_{n,p}(s)$$
	Finally, $(d)$ is immediate from $(c)$ and Proposition \ref{Znpfac}.
\end{proof}

\section{Reciprocity Law for Quadratic Forms}

In this section we first recall the classical proof of the reciprocity law for Gauss sums using the transformation formula of the Jacobi theta function.  Then we generalize the proof to the Riemann theta function, and obtain a reciprocity law for exponential sums of quadratic forms. 

Both proofs are included, since the former is relatively short and may clarify the latter.

\subsection{Jacobi Theta Function}

For $(z,\tau)\in \CC\times \frakH$, the Jacobi theta function is
\begin{align} \vartheta(z,\tau) = \sum_{n\in\ZZ} e(\frac{1}{2}n^2\tau+ nz),\end{align}
and satisfies the functional equation
\begin{align}
\vartheta(\frac{z}{\tau},\frac{-1}{\tau}) &= \sqrt{-i\tau} e(\frac{z^2}{2\tau}) \vartheta(z,\tau).\label{JFE}
\end{align} 
\begin{lemma}\label{Tlem}For any $r\in \RR$, 
	$$\lim\limits_{\tau\rar i\infty}\vartheta(r,\tau)=1.$$
\end{lemma}
\begin{proof}
	We have
	$$\vartheta(r,\tau)-1 = 2\sum_{n=1}^\infty \exp(\frac{1}{2}\pi i n^2 \tau)\cos(2\pi  nr),$$
	so that, for $\tau=u+iv$, 
	$$ |\vartheta(r,\tau)-1|\leq 2\sum_{n=1}^\infty \exp(-\frac{1}{2}\pi n^2 v) < 2\int_{0}^\infty \exp(-\frac{1}{2}\pi x^2v) = \frac{2}{\sqrt{2v}},$$
	which goes to $0$ as $v \rar \infty$.
\end{proof}

For $m>0$ a positive integer and $k\in \ZZ$, we put
\begin{align}\vartheta_{k,m}(z,\tau) = \sum_{n\equiv k(\mod m)} e(\frac{1}{2}n^2\tau+nz),\end{align}
so that
$$ \vartheta(z,\tau) = \sum_{k=0}^{|m|-1} \vartheta_{k,m}(z).$$

\begin{lemma}\label{Tkm}
	$$ \vartheta_{k,m}(z,\tau) = \frac{1}{m}(-i\tau)^{-1/2} e(-\frac{z^2}{2\tau}) \vartheta(\frac{k}{m}+\frac{z}{m\tau},\frac{-1}{m^2\tau}).$$
\end{lemma}
\begin{proof}We have
	\begin{align*}\vartheta_{k,m}(z,\tau) &= \sum_{n\in \ZZ} e(\frac{1}{2}(nm+k)^2\tau+(nm+k)z) = e(\frac{1}{2}k^2\tau + kz)\sum_{n\in \ZZ}e(\frac{1}{2}n^2(m^2\tau) + nm(k\tau+z))\\
	&=e(\frac{1}{2}k^2\tau+kz) \vartheta(m(k\tau+z),m^2\tau).
	\end{align*}
	Now using (\ref{JFE}),
	\begin{align*}\vartheta(m(k\tau+z),m^2 \tau) &= (-im^2\tau)^{-\frac{1}{2}} e(-\frac{(k\tau+z)^2}{2\tau}) \vartheta(\frac{k\tau+z}{m\tau},\frac{-1}{m^2\tau}).\end{align*}
	Then
	$$ \vartheta_{k,m}(z,\tau) = (-im^2\tau)^{-\frac{1}{2}} e(\frac{1}{2}k^2\tau + kz - \frac{(k\tau+z)^2}{2\tau}) \vartheta(\frac{k}{m}+\frac{z}{m\tau},\frac{-1}{m^2\tau}).$$ 
The claim then follows from
	$$ \frac{1}{k^2\tau}+kz-\frac{(k\tau+z)^2}{2\tau})= -\frac{z^2}{2\tau}.$$
\end{proof}

\begin{prop} Let $a,b$ be non-zero and coprime, and $c\in \ZZ$ such that $ab+c$ is even.  Then
	$$ \frac{1}{|b|}\sum_{x=0}^{|b|-1} e(\frac{ax^2+cx}{2b})\vartheta(\frac{x}{b}+z,\tau)  = (\sqrt{it}) e\left(-\frac{s^2b}{2a}\right) \frac{1}{|a|}\sum_{y=0}^{|a|-1} e\left(\frac{-by^2+cy}{2a}\right)\vartheta(\frac{y}{a}+z-\frac{c}{2ab},\tau-\frac{1}{ab}).$$
\end{prop}
\begin{proof}
Let $t=a/b$, $s=c/2b$. The condition on $a$ and $b$ ensures that the function $e(\frac{1}{2}x^2t+sx)$ has period $b$ in $x$. For $z\in \CC$, $\tau\in\frakH$ we have
\begin{align}\begin{split}\vartheta(s+z,t+\tau)&= \sum_{n} e\left(\frac{n^2}{2}t+ns\right)e\left(\frac{n^2}{2}t+nz\right)= \sum_{k(\mod b)} e\left(\frac{k^2}{2}t+ks\right) \vartheta_{k,b}(z,\tau)\\
&= \frac{1}{|b|} (-i\tau)^{-1/2} e(-\frac{z^2}{2\tau}) \sum_{k(\mod b)} e\left(\frac{k^2}{2}t+ks\right) \vartheta(\frac{k}{b} + \frac{z}{b\tau},\frac{-1}{b^2\tau}).\label{F0}\end{split}
\end{align}
Then we can write
\begin{align}\label{F1} \frac{1}{|b|}\sum_{k(\mod b)} e\left(\frac{k^2}{2}t+ks\right) \vartheta(\frac{k}{b}+\frac{z}{b\tau}, -\frac{1}{b^2\tau}) = (-i\tau)^{\frac{1}{2}} e(\frac{z^2}{2\tau}) \vartheta(s+z,t+\tau).\end{align}

Now, by (\ref{JFE})
\begin{align}\label{F2}
\vartheta(s+z,t+\tau) = (-i(t+\tau))^{-1/2}e\left(-\frac{(s+z)^2}{2(t+\tau)}\right) \vartheta(\frac{s+z}{t+\tau},\frac{-1}{t+\tau}).
\end{align}
Letting
$$ s_1 = \frac{s}{t},\ \ t_1 = \frac{-1}{t},\ \ \ z_1 =\frac{zt-s\tau}{t(t+\tau)},\ \ \tau_1=\frac{\tau}{t(t+\tau)},$$
we have
$$ \frac{s+z}{t+\tau} = s_1+ z_1,\ \ \ \frac{-1}{t+\tau} = t_1+\tau_1,\ \ \ \frac{z_1}{\tau_1}=t\frac{z}{\tau}-s.$$
The same way as (\ref{F0}),
\begin{align} \vartheta(\frac{s+z}{t+\tau},\frac{-1}{t+\tau}) = \frac{1}{|a|} (-i\tau_1)^{-1/2} e(\frac{-{z_1}^2}{2\tau_1}) \sum_{r(\mod a)} e\left(\frac{r^2}{2}t_1+rs_1\right) \vartheta(\frac{r}{a} + \frac{z_1}{a\tau_1}, -\frac{1}{a^2\tau_1}).\label{F3}\end{align}
Now, combining (\ref{F1}), (\ref{F2}), (\ref{F3}), we obtain
$$\frac{1}{|b|} \sum_{k(\mod b)} e\left(\frac{k^2t}{2}+ks\right) \vartheta(\frac{k}{b}+\frac{z}{b\tau},-\frac{1}{b^2\tau}) = X \cdot Y. \frac{1}{|a|}\sum_{r(\mod a)} e\left(\frac{r^2t_1}{2}+rs_1\right) \vartheta(\frac{r}{a}+\frac{z_1}{a\tau_1},-\frac{1}{a^2\tau_1}),$$
where
\begin{align*} X &= (-i\tau)^{\frac{1}{2}}(-i(t+\tau))^{-\frac{1}{2}} (-i\tau_1)^{-\frac{1}{2}} = \sqrt{it},\\
Y &= e(\frac{z^2}{2\tau})e\left(-\frac{(s+z)^2}{2(t+\tau)}\right)e(-\frac{z_1^2}{2\tau_1}) = e\left(-\frac{s^2}{2t}\right).
\end{align*}
Replacing $z$ by $\tau z$ and $z_1/\tau_1$ by $tz-s$, 
$$ \frac{1}{|b|}\sum_{k(\mod b)} e\left(\frac{k^2t}{2}+ks\right)\vartheta(\frac{k}{b}+\frac{z}{b},\frac{-1}{b^2\tau}) = (\sqrt{it}) e\left(-\frac{s^2}{2t}\right) \frac{1}{|a|}\sum_{r(\mod a)} e\left(-\frac{r^2}{2t}+\frac{rs}{t}\right)\vartheta(\frac{r}{a} + \frac{tz-s}{a},\frac{-t^2-t\tau}{a^2\tau}).$$
\end{proof}
If we put $\tau = it$ in the statement of the Proposition and let $t\rar \infty$, by Lemma \ref{Tlem} we obtain the reciprocity law for Gauss sums.

\subsection{Riemann Theta Function}

For $({\bf z},\tau)\in\CC^n\times M_n(\CC)$, and ${\bf x}\in \CC^n$, let
\begin{align} \label{fzt}f_{{\bf z},\tau}({\bf x}) = \frac{1}{2}{}^t {\bf x}\tau {\bf x} + {}^t {\bf x} {\bf z},\ \ \ {\bf x}\in \ZZ^n.\end{align}

Let $\frakH_n$ denote the Siegel upper half-space of degree $n$. The Riemann theta function is
$$ \Theta: \CC^n \times \frakH_n \rar \CC, \ \ \ \Theta({\bf z},\tau) = \sum_{{\bf x}\in \ZZ^n} e(f_{{\bf z},\tau}({\bf x}))=\sum_{{\bf x}\in \ZZ^n}\exp(\pi i {}^t{\bf x} \tau {\bf x} + 2\pi i {}^t{\bf x} {\bf z}).$$

It satisfies the transformation formula: 
$$ \Theta({}^t (c\tau+d)^{-1} {\bf z}, \gamma\cdot \tau) = \zeta \det(c\tau+d)^{\frac{1}{2}}e(\frac{1}{2} {}^t {\bf z} (c\tau+d)^{-1}c{\bf z}) \Theta({\bf z},\tau),$$
where 
$$\gamma\cdot \tau = (a\tau + b)(c\tau +d)^{-1}\ \ \ \text{ for } \gamma = \mtwo{a& b \\ c & d}\in \Sp_{2n}(\ZZ),$$
and $\zeta$ is an eighth root of unity that in general depends on $\gamma$ and the choice of the square root $\det(c\tau+d)^{\frac{1}{2}}$. In the special case $\gamma = \mtwo{0 & -1\\ 1 & 0}$ the transformation formula takes the form
\begin{align}\label{RFE} \Theta(\tau^{-1}{\bf z},-\tau^{-1}) = \alpha(\bfz,\tau)\Theta({\bf z},\tau),\end{align}
where
\begin{align}\label{alpha} \alpha(\bfz,\tau) = \det(-i\tau)^{\frac{1}{2}} \exp(i\pi ( {}^t \bfz \tau^{-1} \bfz)).\end{align}
and $\tau \rar \det(-i\tau)^{\frac{1}{2}}$ the unique branch that is positive when $\tau$ is purely imaginary (\cite[p.195]{Mum83I}). This is well-defined since $\frakH_n$ is simply connected.

Our first task is to extend $\tau \rar \det(-i\tau)^{\frac{1}{2}}$ to non-singular $\frakt\in \Sym_n(\RR)$. We write $\sigma_{\frakt}$ for the signature of $\frakt$, so that
$$ \sigma_{\frakt} = \#(\text{positive eigenvalues of }\frakt) - \#(\text{negative eigenvalues of }\frakt),$$
counted with multiplicity.

\begin{lemma}\label{siglem} Let $\frakt \in \Sym_n(\RR)$ be non-singular. Then
	$$\lim_{\epsilon\rar 0^+} \det(-i(\frakt+\epsilon i I_n))^{\frac{1}{2}}=|\det(\frakt)|^{\frac{1}{2}} e^{-i\pi \sigma_{\frakt}/4}.$$
\end{lemma}
\begin{proof} Let $\tau=u+iv\in \frakH_n$, and write $v={}^t a a$ for some $a\in\GL_n(\RR)$. Then ${}^t a^{-1} \tau a^{-1}= u'+iI_n$, where $u'={}^ta^{-1} u a^{-1}$, and $\det(-i\tau) = |\det(a)|^2 \det(I_n-iu')$. Writing $u'=b\delta b^{-1}$ for $b$ unitary and $\delta=\diag(d_1,\cdots, d_n)$ diagonal, $\det(I_n-iu') = \det(I_n - i \delta) = \prod_{j=1}^n (1-id_j).$ For $z=re^{i\theta}$, $\theta\in (-\pi,\pi)$, let us for the moment write $\sqrt{z}=r^{\frac{1}{2}}e^{i\theta/2}$. Then we have
	$$ \det(-i\tau)^{\frac{1}{2}} = |\det(a)| \prod_{j=1}^n \sqrt{1-id_j}.$$
	Now let $\tau = \frakt + \epsilon i I_n$, and $d_1,\cdots, d_n$ be the eigenvalues of $\frakt$ with multiplicity. We have
	$$ \det(-i(\frakt+\epsilon i I_n))^{\frac{1}{2}} = \epsilon^{\frac{n}{2}} \prod_{j=1}^n \sqrt{1 - \frac{i d_j}{\epsilon}}.$$
	Since for any real number $d$, 
	$$ \lim_{\epsilon\rar0^+} \frac{\sqrt{1-\frac{id}{\epsilon}}}{\sqrt{|1-\frac{id}{\epsilon}|}}= e^{-i\pi \sgn(d)}/4,$$
	we have
	$$ \lim_{\epsilon \rar 0^+} \det(-i(\frakt+\epsilon iI_n))^{\frac{1}{2}} = \prod_{j=1}^n |1-id_j|^{\frac{1}{2}} e^{-i\pi \sgn(d_j)/4}= |\det(\frakt)|^{\frac{1}{2}} e^{-i\pi \sigma_{\frakt}/4}.$$
\end{proof}

\begin{lemma}\label{limlem}Let ${\bfz}\in \RR^n$, and ${}^t u = u\in M_n(\RR)$. Then
	$$ \lim_{t \rar \infty} \Theta(\bfz,u+ it I_n) = 1.$$
\end{lemma}
\begin{proof}We have
	$$ \Theta(\bfz,u+it I_n)-1 = \sum_{\tiny\begin{array}{c}{\bf x}\in \ZZ^n,\\{\bf x}\neq 0\end{array}} \exp(-\pi {}^t{\bf x} {\bf x}t) e(\frac{1}{2} {}^t {\bf x} u {\bf x} + {}^t{\bf x} \bfz),$$
	so that
	\begin{align*} |\Theta(\bfz,u+it I_n)-1| \leq -1+\sum_{x_1,\cdots,x_n=-\infty}^\infty \exp(-\pi t\sum_{j=1}^n x_j^2)= (1 + 2\sum_{x=1}^\infty \exp(-\pi x^2 t))^n -1.
	\end{align*}
	For $t>0$, 
	$$ \sum_{x=1}^\infty \exp(-\pi x^2 t) < \int_{0}^\infty \exp(-\pi x^2) = \frac{1}{\sqrt{2t}}.$$
	Therefore
	$$ \lim_{t\rar \infty}|\Theta(\bfz,u+itI_n) -1| \leq \lim_{t\rar\infty}(1-\frac{2}{\sqrt{2t}})^n -1 = 0 .$$
\end{proof}

For $\bfx\in \ZZ^n$, and $L\subset \ZZ^n$ a sublattice, not necessarily of full rank, put
$$ \Theta_{{\bf x},L}(\bfz,\tau) = \sum_{\bfy\in {\bfx}+L} e(f_{\bfz,\tau}(\bfy)).$$
If $\rank(L)=r\leq n$, there exists $A\in M_{n\times r}(\ZZ)$ of rank $r$ such that 
$$ \ZZ^r \rar L,\ \ \ {\bf x} \mapsto A {\bf x}$$
is a bijection.

\begin{lemma}\label{parthet} $L=A \ZZ^r$ be a lattice of rank $r$, with $A\in M_{n\times r}(\ZZ)$ having full rank, as above. Then
	\begin{itemize}
		\item[(a)] Writing $\Theta^{(r)}$ for the Riemann theta function on $\CC^r\times \frakH_r$, 
		$$\Theta_{{\bf x},L}(\bfz,\tau) = e(f_{\bfz,\tau}({\bf x})) \Theta^{(r)}({}^t A (\bfz+\tau {\bf x}), {}^tA\tau A).$$
		\item[(b)] If $L$ has rank $n$, so that $A\in M_n(A)$ and $\det(A)\neq 0$,
		$$ \Theta_{{\bfx},L}(\bfz,\tau) = |\det(A)|^{-1} \det(-i\tau)^{-\frac{1}{2}} \exp(-\pi i {}^t \bfz \tau^{-1} \bfz) \Theta(A^{-1}(\tau^{-1}\bfz+{\bfx}), -({}^t A\tau A)^{-1}).$$  
	\end{itemize}
\end{lemma}
\begin{proof}
	Part $(a)$ follows immediately from
	$$ f_{z,\tau}({\bfx}+A\bfy) = f_{z,\tau}({\bf x}) + f_{z,\tau}(A\bfy) + {}^t \bfx A \bfy = f_{z,\tau}({\bf x}) + f_{{}^tA\bfx, {}^t A \tau A}({\bf y}).$$
	For $(b)$, we put
	$$z_0 = {}^tA(z+\tau {\bf x}),\ \ \tau_0 = {}^tA \tau A.$$
	and use the transformation formula  (\ref{RFE}) to write
	$$ \Theta({}^t A(\bfz + \tau {\bf x}),{}^t A \tau A)= \det(-i\tau_0)^{-\frac{1}{2}} \exp(-i \pi {}^t \bfz_0 \tau_0^{-1} \bfz_0) \Theta(\tau_0^{-1}\bfz_0,-\tau_0^{-1}).$$
	Then $(b)$ follows from $(a)$ after checking that 
	$$ \det(-i\tau_0) = \det(A)^2 \det(-i\tau),\ \ \ {}^t \bfz_0 \tau_0^{-1}\bfz_0 = {}^t \bfz \tau^{-1} \bfz + 2f_{\bfz,\tau}({\bfx}),\ \ \ \tau_0^{-1}\bfz_0 = A^{-1}(\tau^{-1}\bfz+{\bfx}),$$
	and noting $e(f_{\bfz,\tau}({\bfx}))$ cancels with the $\exp(-\pi i 2f_{\bfz,t}({\bfx}))$ factor coming from ${}^t \bfz_0 \tau_0^{-1} \bfz_0$.
\end{proof}
Since $f_{\bfz,\tau}({\bf x})$ depends linearly on $(\bfz,\tau)$, if $\bfz_1,\bfz_2\in \CC^n$, and $\tau_1,\tau_2\in M_n(\CC)$ with $\tau_1+\tau_2\in \frakH_n$, we have
$$ \Theta(\bfz_1+\bfz_2,\tau_1+\tau_2) = \sum_{{\bf x}\in \ZZ^n}e(f_{\bfz_1,\tau_1}({\bf x}))e(f_{\bfz_2,\tau_2}({\bf x})).$$

For $(\frakz,\frakt)\in \CC^n\times \Sym_n(\QQ)$, denote by $L_{\frakz,\frakt}$ the set of ${\bf y}\in \ZZ^n$ such that
\begin{align}\label{Lzt} f_{\frakz,\frakt}({\bf y}) + {}^t {\bf x} \frakt {\bf y}\in \ZZ,\ \ \text{ for all }{\bf x}\in \ZZ^n.\end{align}
Recall that an integral symmetric matrix is called \textit{even} if its diagonal entries are even.
\begin{lemma}\label{intlem}
	Suppose that $(\bfz,\tau)\in \CC^n\times\Sym_n(\ZZ)$. Then $L_{\bfz,\tau}=\ZZ^n$ if and only $\tau+\diag(2\bfz)$ is an even integral matrix.
\end{lemma}
\begin{proof}
	Since $\tau\in \Sym_n(\ZZ)$, we have ${}^t\bfx \tau \bfx\in \ZZ$ for any $\bfx\in \ZZ^n$. Then $f_{\bfz,\tau}(\bfx)\in \ZZ$ only if ${}^t\bfx \bfz \in \frac{1}{2}\ZZ$ for all $\bfx\in\ZZ^n$, which is to say $2\bfz\in \ZZ^n$. Now 
	$$ f_{\bfz,\tau}(\bfx+\bfy)= f_{\bfz,\tau}(\bfx) + f_{\bfz,\tau}(\bfy) + {}^t\bfx\frakt \bfy,\ \ \ f_{\bfz,\tau}(-\bfx) = f_{\bfz,\tau}(\bfx) - {}^t\bfx (2\bfz).$$
	It follows that, assuming $2\bfz\in \ZZ^n$, $f_{\bfz,\tau}(\bfx)\in\ZZ$ for all $\bfx\in \ZZ^n$ if and only if the same holds for $\bfx=e_1,\cdots, e_n$. On the other hand,
	$$ f_{\bfz,\tau}(e_i) = {}^te_i T e_i,\ \ \text{ where }T=\frakt + \diag(2\bfz),$$
	so $L_{\bfz,\tau}=\ZZ^n$ if and only if the integral matrix $T$ is even. Note that, since $\tau\in\Sym_n(\ZZ)$, if $T$ is even $2\bfz\in \ZZ^n$ automatically.
\end{proof}

Suppose $\tau\in \Sym_n(\QQ)$. Then we may choose $U$, $V \in \GL_n(\ZZ)$ such that $\tau=UDV^{-1}$, with 
$$D=\diag(\frac{p_1}{q_1},\cdots, \frac{p_n}{q_n}),$$
and $p_i,q_i$ integers such that $q_i>0$ and $\gcd(p_i,q_i)=1$. Putting
$$ P = \diag(p_1,\cdots, p_n),\ \ \ Q = \diag(q_1,\dots, q_n),$$
and
\begin{align}\label{PQ}A=UP,\ \ B = VQ,\end{align}
we have
\begin{align}\label{AB} \tau = AB^{-1}.\end{align}

\begin{defn}For $\tau\in \Sym_n(\QQ)$, we say $\tau=AB^{-1}$ is in \textit{reduced form}, if $A=UP$, $B=VQ$, with $U$, $V$, $P$, $Q$ as above.
\end{defn}

Thus any $\tau\in \Sym_n(\QQ)$ has a (non-unique) reduced form.

\begin{prop}\label{Lprop}Let $\tau\in \Sym_n(\QQ)$, $\bfz\in \CC^n$, and write $\tau=AB^{-1}$ in reduced form. 	
	\begin{itemize}
		\item[(a)] $\bfy,\ \tau \bfy\in\ZZ^n$ if and only if $\bfy\in B\ZZ^n$.
		\item[(b)] $L_{\bfz,\tau}=\{\bfy\in B\ZZ^n: f_{\bfz,\tau}(\bfy)\in \ZZ\}.$
		\item[(c)] $L_{\bfz,\tau}=B\ZZ^n$ if and only if ${}^tBA + 2\diag({}^tB\bfz)$ is an even integral matrix.
	\end{itemize}
\end{prop}

\begin{proof}
	With $U$, $V$, $P$, $Q$ as in (\ref{PQ}), we have $\tau= UPQ^{-1} V^{-1}$. Since $U \in \GL_n(\ZZ)$, $\tau \bfy\in \ZZ^n$ if and only if $PQ^{-1}V^{-1}\bfy\in \ZZ^n$. Since $V\in \GL_n(\ZZ)$, $\bfy\in \ZZ^n$ if and only if $\bfy_0=V^{-1}\bfy\in \ZZ^n$. Therefore $\bfy,\tau \bfy\in \ZZ^n$ if and only if $\bfy=V\bfy_0$ for some $\bfy_0\in \ZZ^n$ such that $PQ^{-1} \bfy_0\in \ZZ^n$. Since $P$, $Q$ are integral diagonal matrices with $\gcd(p_i,q_i)=1$, 
	$$PQ^{-1}\bfy_0\in \ZZ^n \iff Q^{-1} \bfy_0\in \ZZ^n \iff  \bfy \in B\ZZ^n,$$
	which proves $(a)$.
	
	Suppose $\bfy\in L_{\bfz,\tau}$. Setting $\bfx=0$ in (\ref{Lzt}) we obtain $f_{\bfz,\tau}(\bfy)\in \ZZ$. Subtracting from $(\ref{Lzt})$ for an arbitrary $\bfx$ we obtain ${}^t\bfx \frakt \bfy\in\ZZ$ for all $\bfx$, which is to say $\tau \bfy\in \ZZ^n$. Conversely, if $f_{\bfz,\tau}(\bfy)\in \ZZ$ and $\bfy, \tau \bfy\in \ZZ^n$, we have $\bfy\in L_{\bfz,\tau}$. Then $(b)$ follows from $(a)$.
	
	It follows from $(b)$ that $L_{\bfz,\tau}=B\ZZ^n$ if and only if $f_{\bfz,\tau}(B\bfy_0)\in\ZZ$ for all $\bfy_0\in \ZZ$. Now
	$$f_{\bfz,\tau}(B\bfy_0) = \frac{1}{2} {}^t \bfy_0 {}^tB A\bfy_0 + {}^t\bfy_0 {}^tB \bfz = f_{\bfz_0,\tau_0}(\bfy_0),\ \ \ \text{ where }\bfz_0 = {}^tB \bfz,\ \ \tau_0={}^tB A.$$
	Note $\tau\in \Sym_n(\QQ)$ implies ${}^tB A$ is integral and symmetric. Then using Lemma \ref{intlem}, $f_{\bfz,\tau}(B\ZZ^n) \subset \ZZ$ if and only if $\tau_0 + 2\diag(\bfz_0) = {}^tBA + 2\diag({}^t B\bfz)$ is integral and even.
\end{proof}

Now assume that $(\frakz,\frakt)\in \frac{1}{2}\ZZ^n\times \Sym_n(\QQ)$, and that $\frakt=AB^{-1}$ is in reduced form, with ${}^tB A + 2\diag(B\frakz)$ an even integral matrix, so that $L_{\frakz,\frakt}=B\ZZ^n$ by Proposition \ref{Lprop}(c). For $(\bfz,\tau)\in \CC^n\times \frakH_n$, we have
$$ \Theta(\frakz+\bfz,\frakt+t) = \sum_{\bfx\in \ZZ^n} e(f_{\frakz,\frakt}(\bfx))e(f_{\bfz,\tau}(\bfx)) = \sum_{{\bf x}\in \ZZ^n(\mod B)} e(f_{\frakz,\frakt}({\bf x})) \Theta_{{\bf x},L_{\frakz,\frakt}}(z,\tau).$$
Then using Lemma \ref{parthet}(b), 
\begin{align}\label{E1}\Theta(\frakz+\bfz,\frakt+\tau) = |\det(B)|^{-1} \alpha(\bfz,\tau)^{-1}\sum_{{\bf x}\in \ZZ^n(\mod B)} e(f_{\frakz,\frakt}({\bf x}) \Theta(B^{-1}(\tau^{-1}\bfz + {\bf x}), -({}^t B \tau B)^{-1}),\end{align}
where $(\alpha(\bfz,\tau)$ is as in (\ref{alpha}). Applying (\ref{RFE}), we have
\begin{align}\label{E2}\Theta(\frakz+\bfz, \frakt+\tau) = \alpha(\frakz+\bfz,\frakt+\tau)^{-1}\Theta((\frakt+\tau)^{-1}(\frakz+\bfz), -(\frakt+\tau)^{-1}).\end{align}
Now we write
\begin{align}\label{z1t1}(\frakt+\tau)^{-1}(\frakz+\bfz) = \frakz_1+\bfz_1,\ \ \ -(\frakt+\tau)^{-1}= \frakt_1+\tau_1,\end{align}
where
\begin{align}\label{z1t2} \frakz_1 = \frakt^{-1}\frakz,\ \ \ \frakt_1 = -\frakt^{-1}.\end{align}
Since $\tau = AB^{-1}$, we have $\tau_1 = -BA^{-1}$. Now, in general
$$ {}^tA B + 2\diag({}^t A \frakz_1) = {}^t(B {}^t A + 2\diag({}^t B \frakz)) = BA^t + 2\diag({}^t B \frakz).$$
Since the right-hand side is by assumption an even integral matrix, so is the left-hand side, hence $L_{\frakz_1,\frakt_1} = A\ZZ^n$ by Proposition \ref{Lprop}(c). Therefore as in \ref{E1}, we have
\begin{align}\label{E3} \Theta(\frakz_1+\bfz_1, \frakt_1+\tau_1) = |\det(A)|^{-1} \alpha(\bfz_1,\tau_1)^{-1} \sum_{{\bf x}_1\in \ZZ^n(\mod A)} e(f_{\frakz_1,\frakt_1}({\bf x}_1)) \Theta(A^{-1}(\tau_1^{-1}\bfz_1 + {\bf x}_1),-({}^t A \tau_1 A)^{-1}).\end{align}
Combining with (\ref{E1}) and (\ref{E2}), we arrive at the preliminary form of our main identity:
\begin{align} \begin{split}\label{prelimid}&\sum_{{\bf x}\in \ZZ^n(\mod B)} e(f_{\frakz,\frakt}({\bf x})) \Theta(B^{-1}(\tau^{-1}\bfz+{\bf x}), -({}^t B\tau B)^{-1}) \\
= F(\frakz,\frakt,\bfz,\tau) &\sum_{{\bf x}_1\in \ZZ^n(\mod A)} e(f_{\frakz_1,\frakt_1}({\bf x}_1)) \Theta(A^{-1}(\tau_1^{-1}\bfz_1 + {\bf x}_1), -({}^t A\tau_1 A)^{-1}),
\end{split}\end{align} 
where 
\begin{align}\label{factor}F(\frakz,\frakt,\bfz,\tau)= \frac{|\det(B)|}{|\det(A)|} \alpha(z,\tau) \alpha(\frakz+\bfz,\frakt+\tau)^{-1} \alpha(\bfz_1,\tau_1)^{-1}.\end{align}

\begin{lemma}\label{idlem}With $\tau_1$, $\bfz_1$ as in (\ref{z1t2}), the following identities holds:
	$$\begin{array}{cll}
	(a)  & \tau_1 = -(\frakt+\tau)^{-1}+\frakt^{-1} = (\frakt+\tau)^{-1}\tau \frakt^{-1}= \frakt^{-1}\tau (\frakt+\tau)^{-1}, \\
	(b)  &\bfz_1 = (\frakt+\tau)^{-1}(\bfz-\tau \frakt^{-1}\frakz).\\
	(c) & \frakt = \tau (\frakt+\tau)^{-1} \tau_1^{-1} \\
	(d) &{}^t \bfz \tau^{-1}\bfz +{}^t \frakz \frakt^{-1} \frakz = {}^t(\frakz+\bfz)(\frakt+\tau)^{-1}(\frakz+\bfz)+{}^t \bfz_1 \tau_1^{-1} \bfz_1.\\
	(e) & \alpha(\bfz,\tau)\alpha(\frakz+\bfz,\frakt+\tau)^{-1}\alpha(\bfz_1,\tau_1)^{-1} = |\det(\frakt)|^{\frac{1}{2}} e^{i\pi\sigma_{\frakt}/4}\exp(-i\pi{}^t \frakz \frakt^{-1}\frakz).
	\end{array}$$
\end{lemma}
\begin{proof}
	Parts $(a)$, $(b)$ follow immediately from the defining equations (\ref{z1t1}), and $(c)$ follows from $(a)$.
	For $(d)$, first we use (\ref{z1t1}) to write
	\begin{align*} {}^t(\frakz+\bfz)(\frakt+\tau)^{-1}(\frakz+\bfz) = {}^t(\frakz+\bfz)(\frakz_1+\bfz_1) = ({}^t \frakz + {}^t \bfz)(\frakt^{-1}\frakz+\bfz_1)={}^t\frakz\frakt^{-1}\frakz + {}^t\frakz \bfz_1 + {}^t \bfz \frakt^{-1}\frakz + {}^t \bfz \bfz_1.\end{align*}
	Then using $(a)$ and $(b)$ we also write
	\begin{align*}{}^t \bfz_1 \tau_1^{-1} \bfz_1 = ({}^t \bfz - {}^t \frakz \frakt^{-1} \tau)\tau^{-1}\frakt \bfz_1 = {}^t \bfz \tau^{-1}\frakt \bfz_1 - {}^t \frakz \bfz_1.\end{align*}
	Summing the two relations above, we obtain
	$$ {}^t(\frakz+\bfz)(\frakt+\tau)^{-1}(\frakz+\bfz) + {}^t \bfz_1\tau_1^{-1} \bfz_1 = {}^t \frakz \frakt^{-1}\frakz + {}^t \bfz \tau^{-1}(\tau\frakt^{-1}\frakz + \tau \bfz_1 + \frakt \bfz_1) = {}^t \frakz \frakt^{-1}\frakz + {}^t \bfz \tau^{-1} \bfz,$$
	with the last equality using $(b)$. This proves $(d)$. 
	
	The left-hand side of $(e)$ is equal to the product $X\cdot Y$ where 
	$$X=\det(-i\tau)^{\frac{1}{2}}\det(-i(\frakt+\tau))^{-\frac{1}{2}} \det(-i\tau_1)^{-\frac{1}{2}}=\det(i\frakt)^{\frac{1}{2}}=|\det(\frakt)| e^{i\pi \sigma_{\frakt}/4}$$
	by $(c)$ and Lemma \ref{siglem}, and
	$$
	Y=\exp(-i\pi \left({}^t\bfz \tau^{-1}\bfz -{}^t(\frakz+\bfz)(\frakt+\tau)^{-1}(\frakz+\bfz)-{}^t \bfz_1 \tau_1^{-1} \bfz_1\right))=\exp(-i\pi {}^t \frakz \frakt^{-1}\frakz)
	$$
	by $(d)$. This proves $(e)$.
\end{proof}

It follows from part $(e)$ of the Lemma that 
\begin{align*}F(\frakz,\frakt,\bfz,\tau) = |\det(\frakt)|^{-1} |\det(\frakt)|^{\frac{1}{2}} e^{i\pi \sigma_{\frakt}/4} \exp(-i\pi {}^t\frakz \frakt^{-1} \frakz)=\det(-i\frakt)^{-\frac{1}{2}}\exp(-i\pi {}^t \frakz\frakt^{-1})=\alpha(\frakz,\frakt)^{-1}.\end{align*}

The main theorem is now as follows.

\begin{thm}\label{thmB}Let $\frakt\in \Sym_n(\QQ)$ be non-singular, write $\frakt=AB^{-1}$ in reduced form, and put $\frakn={}^tB A$. Let $\frakc\in \QQ^n$ be such that $\frakn+2\diag(\frakc)$ is integral and even. Then for all $(\bfz,\tau)\in \CC^n\times \frakH_n$,
	$$ \sum_{{\bf x}\in B^{-1}\ZZ^n(\mod \ZZ^n)}e(f_{\frakn,\frakc}({\bf x}))\Theta(\bfz+{\bf x},\tau) = \det(i\frakt)^{-\frac{1}{2}} e^{-i\pi{}^t(\frakc \frakn^{-1}\frakc)}\sum_{{\bf x}_1\in A^{-1}\ZZ^n(\mod \ZZ^n)} e(f_{-\frakn,\frakc}({\bf x}_1))\Theta(\bfz+{\bf x}_1 - \frakn^{-1},\tau-\frakn^{-1}).$$ 
\end{thm}
\begin{proof}
	Let $\frakz = {}^t B^{-1}\frakc$, so that $(\frakz,\frakt)$ satisfies the condition in Proposition \ref{Lprop}(c), and (\ref{prelimid}) holds. Replacing $\bfz$ by $\tau B \bfz$ and ${\bf x}\in \ZZ^n(\mod B)$ with $B^{-1}{\bf x}\in B\ZZ^n(\mod \ZZ^n)$, the left-hand side of (\ref{prelimid}) becomes
	$$ \sum_{{\bf x}\in B^{-1}\ZZ^n(\mod \ZZ^n)} e(f_{\frakz,\frakt}(B{\bf x}))\Theta(\bfz+{\bf x}), -({}^t B \tau B)^{-1}).$$
	
	Using Lemma \ref{idlem}, we have
	$$ \tau_1^{-1}\bfz_1 = \frakt \tau^{-1}(\bfz-\tau\frakt^{-1}\frakz)=\frakt \tau^{-1} z-\frakz \so A^{-1}(\tau^{-1}_1\bfz_1 + {\bf x}_1) =B^{-1} \tau^{-1}\bfz - A^{-1}\frakz +A^{-1}{\bf x}_1.$$
	which becomes $\bfz - A^{-1}\frakz - {\bf x}_1 = A - \frakn^{-1}\frakc - {\bf x}$ after the replacing $\bfz$ by $\tau B\bfz$, and ${\bf x}_1\in \ZZ^n(\mod A\ZZ^n)$ with $A^{-1}{\bf x}_1 \in A^{-1}\ZZ^n(\mod \ZZ^n)$. 
	
The right-hand side of (\ref{prelimid}) is then
	$$ \alpha(\frakz,\frakt)^{-1} \sum_{{\bf x}_1\in A^{-1}\ZZ^n(\mod \ZZ^n)} e(f_{\frakz,\frakt}(A{\bf x}))\Theta(z+{\bf x}_1-\frakn^{-1}\frakc, -({}^t A\tau_1 A)^{-1}).$$
	Now we have
	$$ f_{\frakz,\frakt}(B{\bf x}) = f_{\frakc,\frakn}({\bf x}),\ \ \ f_{\frakz_1,\frakt_1}(A{\bf x}_1)= f_{\frakc,-\frakn}({\bf x}_1),\ \ {}^t \frakz \frakt^{-1}\frakz = {}^t \frakc \frakn^{-1} \frakc,$$
	and $\sigma_{\frakt}=\sigma_{\frakn}$ since $\frakn = {}^t B \frakt B$. Therefore
	$$ F(\frakz,\frakt,\bfz,\tau)=\alpha(\frakz,\frakt)^{-1} = |\det(\frakt)|^{-\frac{1}{2}} e^{i\pi \sigma_{\frakn}/4} \exp(\pi i {}^t \frakc \frakn^{-1}\frakc).$$
	The identity (\ref{prelimid}) then becomes
	$$ \sum_{{\bf x}\in B^{-1}\ZZ^n(\mod \ZZ^n)}e(f_{\frakn,\frakc}({\bf x}))\Theta(z+{\bf x},\tau_B) = \det(i\frakt)^{-\frac{1}{2}} \exp(-i\pi {}^t\frakc \frakn^{-1}\frakc) \sum_{{\bf x}_1\in A^{-1}\ZZ^n(\mod \ZZ^n)} e(f_{-\frakn,\frakc}({\bf x}))\Theta(z +{\bf x}_1 - \frakn^{-1}\frakc,\tau_A),$$ 
	where
	$$ \tau_B = -({}^t B \tau B)^{-1},\ \ \ \tau_A = -({}^t A \tau_1 A)^{-1} = \tau_B - \frakn^{-1}.$$
	The statement of the theorem then results after replacing $\tau_B$ by $\tau$.
\end{proof}

\begin{cor}[Reciprocity]\label{GR} Let $\frakt\in \Sym_n(\QQ)$ be non-singular, write $\frakt=AB^{-1}$ in reduced form, and put $\frakn={}^tB A$. Let $\frakc\in \QQ^n$ be such that $\frakn+2\diag(\frakc)$ is integral and even. Then
	$$ \frac{1}{\sqrt{|\det(B)|}}\sum_{{\bf x}\in B^{-1}\ZZ^n\ (\mod \ZZ^n)} e(\frac{1}{2} {}^t {\bf x} \frakn {\bf x} + {}^t {\bf x} \frakc)= \exp(-\pi i{}^t\frakc \frakn^{-1} \frakc)\frac{e^{i\pi \sigma_{\frakn}/4}}{\sqrt{|\det(A)|}} \sum_{{\bf x}\in A^{-1}\ZZ^n (\mod \ZZ^n)} e(-\frac{1}{2}{}^t{\bf x} \frakn {\bf x} + {}^t{\bf x} \frakc).$$
\end{cor}
\begin{proof}Put $\tau=it\in i\RR$ in the statement of the theorem and take the limit as $t\rar\infty$. By Lemma \ref{limlem} all the theta values go to $1$. Write
	$$ (\det(-i\frakt)^{-\frac{1}{2}} =\frac{|\det(B)|^{\frac{1}{2}}}{|\det(A)|^{\frac{1}{2}}} e^{-i \pi \sigma_{\frakn}/4}$$
	by Lemma \ref{siglem}, then replace $\frakn$ by $-\frakn$ and correspondingly $\sigma_{\frakn}$ by $-\sigma_{\frakn}$.
\end{proof}

An equivalent formulation was stated as Theorem B in the introduction. The special case $\frakc=0$ generalizes the Landsberg-Schaar relation. 

\begin{cor}\label{GLS}Let $\frakt\in \Sym_n(\QQ)$ be non-singular, write $\frakt =AB^{-1}$ in reduced form, and assume $\frakn={}^tBA$ is even. Then
	$$ \frac{1}{\sqrt{|\det(B)|}}\sum_{{\bf x}\in \ZZ^n(\mod B)} e^{\pi i  {}^t {\bf x} \frakt {\bf x}} = \frac{e^{i\pi \sigma_{\frakt}/4}}{\sqrt{|\det(A)|}} \sum_{{\bf x}\in \ZZ^n(\mod A)} e^{-\pi i {}^t {\bf x} \frakt^{-1} {\bf x}}.$$
\end{cor}

\bibliographystyle{alpha}
\bibliography{refdb}

\end{document}